\pdfoutput=1
\documentclass[11pt,a4paper]{article}
\usepackage[utf8]{inputenc}
\usepackage{amsthm}
\usepackage{amsfonts}
\usepackage{amsmath}
\usepackage{amssymb}
\usepackage{colonequals}
\usepackage{epsfig}
\usepackage{url}
\usepackage{hyperref}
\usepackage{asymptote}
\newtheorem{theorem}{Theorem}
\newtheorem{corollary}[theorem]{Corollary}

\newtheorem{lemma}[theorem]{Lemma}

\newtheorem{observation}[theorem]{Observation}
\newtheorem{question}[theorem]{Question}

\newcommand{\GG}{\mathcal{G}}

\begin{asydef}
usepackage("amsmath");
usepackage("amssymb");
usepackage("xcolor");
unitsize(9mm);

DefaultHead.size=new real(pen p=currentpen) {return 2mm;};

void vertex(pair a, pen barva=white, real pol = 0.1)
{
  filldraw(circle (a, pol), fillpen=barva);
}

pair v[];
int i,j;
\end{asydef}

\begin{document}
\title{Weak diameter coloring of graphs on surfaces}
\author{Zden\v{e}k Dvo\v{r}\'ak
	\thanks{Charles University, Prague, Czech Republic. E-mail: \protect\href{mailto:rakdver@iuuk.mff.cuni.cz}{\protect\nolinkurl{rakdver@iuuk.mff.cuni.cz}}.
	Supported by the ERC-CZ project LL2005 (Algorithms and complexity within and beyond bounded expansion) of the Ministry of Education of Czech Republic.}
		\and Sergey Norin \thanks{McGill University, Montr\'{e}al, Quebec, Canada.  E-mail: \protect\href{mailto: 
			sergey.norin@mcgill.ca}{\protect\nolinkurl{sergey.norin@mcgill.ca}}. Supported by an NSERC Discovery grant.}
		}

\date{}

\maketitle

\begin{abstract}
Consider a graph $G$ drawn on a fixed surface, and assign to each vertex a list of colors of size
at least two if $G$ is triangle-free and at least three otherwise.
We prove that we can give each vertex a color from its list so that
each monochromatic connected subgraph has bounded weak diameter (i.e., diameter
measured in the metric of the whole graph $G$, not just the subgraph).  In case
that $G$ has bounded maximum degree, this implies that each connected monochromatic
subgraph has bounded size.  This solves a problem of Esperet and Joret for planar triangle-free
graphs, and extends known results in the general case to the list setting, answering a question of Wood.
\end{abstract}

The colorings in this paper are not necessarily proper.
The \emph{weak diameter} of a subgraph $H$ of a graph $G$ is the maximum distance
in $G$ between vertices of $V(H)$; note that the distances are measured in the whole graph $G$, not
in the subgraph $H$.  For a non-negative integer $\ell$, a \emph{weak diameter-$\ell$ coloring} of a graph $G$
is an assignment of colors to its vertices such that each monochromatic connected subgraph
has weak diameter at most $\ell$ (in particular, a weak diameter-$0$ coloring is just a proper coloring).
We say a class of graphs $\GG$ has \emph{weak diameter chromatic number} at most $k$
if for some~$\ell$, every graph in $\GG$ has a weak diameter-$\ell$ coloring using at most $k$ colors.

Weak diameter coloring arises in the context of \emph{asymptotic dimension} of graph classes:
Denoting by $G^r$ the graph obtained from $G$ by joining by an edge each pair of distinct vertices at distance at most $r$,
the class $\GG$ has \emph{asymptotic dimension} at most $d$ if for every $r\ge 1$, the class $\{G^r:G\in\GG\}$
has weak diameter chromatic number at most $d+1$.  In a recent breakthrough, Bonamy et al.~\cite{bonamy2021asymptotic}
proved that graphs of bounded treewidth have asymptotic dimension 1 and all proper minor-closed classes have
asymptotic dimension at most~2. As a special case:
\begin{theorem}[Bonamy et al.~\cite{bonamy2021asymptotic}]\label{thm-wdcr}
For any surface $\Sigma$, the class of graphs drawn on $\Sigma$ has weak diameter chromatic number at most $3$.
\end{theorem}

In graphs with bounded maximum degree, the notion of weak diameter coloring coincides with the
well-studied notion of \emph{clustered coloring}.  For a positive integer $s$, a \emph{coloring with clustering $s$}
is an assignment of colors to vertices such that each monochromatic component has size at most $s$ (a~coloring with clustering $1$
is just a proper coloring).  A class of graphs $\GG$ has \emph{clustered chromatic number} at most $k$
if for some $s$, every graph in $\GG$ has a coloring with clustering at most $s$ using at most $k$ colors. We refer the reader to  an extensive survey by Wood~\cite{wood2018defective} for further background on clustered coloring.

\begin{observation}\label{obs-coinc}
Let $k\ge 1$ be an integer and let $\GG$ be a class of graphs.
\begin{itemize}
\item If $\GG$ has clustered chromatic number at most $k$, then $\GG$ also has weak diameter chromatic number at most $k$.
\item If $\GG$ has weak diameter chromatic number at most $k$ and bounded maximum degree, then
$\GG$ also has clustered chromatic number at most $k$.
\end{itemize}
\end{observation}

Hence, by Theorem~\ref{thm-wdcr}, the class of graphs of bounded maximum degree drawn on any fixed surface~$\Sigma$ has clustered chromatic number at most three.
This was proved earlier by Esperet and Joret~\cite{espjor}, who also asked the following question.
\begin{question}\label{quest-trfree}
Consider a surface $\Sigma$ and a positive integer $\Delta$.  Is it true that the class of triangle-free
graphs of maximum degree $\Delta$ than can be drawn on~$\Sigma$ has clustered chromatic number at most two?
\end{question}

It is also natural to consider the list versions of these notions.  Given an~assignment $L$ of lists of colors to vertices
of a graph $G$, a coloring of $G$ is an \emph{$L$-coloring} if the color of each vertex $v\in V(G)$ belongs to $L(v)$.
We say that a class of graphs $\GG$ has \emph{weak diameter choosability} at most $k$ if for some~$\ell$,
every graph in $\GG$ has a weak diameter-$\ell$ $L$-coloring from any assignment $L$ of lists of size at least $k$;
and \emph{clustered choosability} at most $k$ if for some $s$, every graph in $\GG$ has an $L$-coloring with clustering at most $s$
from any assignment $L$ of lists of size at least $k$.

For any surface $\Sigma$, the class of graphs drawn on $\Sigma$ has clustered choosability at most four~\cite{islands}.
This bound cannot be improved (even for non-list coloring), as for every $s$, there exists a planar graph
that has no coloring with clustering at most $s$ using at most three colors.  However,
it has been asked by Wood whether this can be improved for graphs of bounded maximum degree (matching the
result of Esperet and Joret~\cite{espjor} in the non-list setting).
\begin{question}[{Wood~\cite[Open Problem 18]{wood2018defective}}]\label{quest-list}
Consider a surface $\Sigma$ and a positive integer $\Delta$.  Is it true that the class of 
graphs of maximum degree $\Delta$ than can be drawn on $\Sigma$ has clustered choosability at most three?
\end{question}

We answer Questions~\ref{quest-trfree} and~\ref{quest-list} in positive, in the more general setting
of weak diameter choosability.

\begin{theorem}\label{thm-main}
For every surface $\Sigma$,
\begin{description}
\item[(i)] the class of graphs drawn on $\Sigma$ has weak diameter choosability at most three, and
\item[(ii)] the class of triangle-free graphs drawn on $\Sigma$ has weak diameter choosability at most two.
\end{description}
\end{theorem}

It is tempting to ask whether Theorem~\ref{thm-main} could be strengthened by replacing weak diameter with diameter,
where we measure the diameter of each monochromatic connected subgraph inside the subgraph (not
in the ambient graph).  However, such a strengthening of both (i) and (ii) is false even for the class of planar graphs and non-list coloring, as we show in Section~\ref{sec-cex}.

We suspect the following substantial relaxation of the ``triangle-free'' assumption could be sufficient to
make coloring by two colors possible.
\begin{question}\label{quest-near}
For a non-negative integer $r$, let $\GG_r$ be the class of plane graphs with no separating triangles such
that each vertex is at distance at most $r$ from a face of length at least four.  Does $\GG_r$ have
weak diameter chromatic number at most two?
\end{question}
The importance of the presence of non-triangular faces can be seen from the following standard example:
Let $G_n$ be the $n\times n$ grid with a diagonal added to each face (so that all faces except for the outer
one are triangles). By the HEX lemma, any coloring of $G_n$ by colors red and blue contains a red path from the left side
to the right side of the grid, or a blue path from the top side to the bottom side of the grid.
Hence, any coloring of $G_n$ by two colors contains a monochromatic component of weak diameter at least $n-1$.

Moreover, let us remark that Question~\ref{quest-near} has positive answer in the case that the graph has exactly one
non-triangular face, since plane graphs where all vertices are at a bounded distance from a fixed face have bounded treewidth~\cite{rs3}
and consequently asymptotic dimension one~\cite{bonamy2021asymptotic}.

The rest of the paper is organized as follows.  In Section~\ref{sec-cex}, we show that weak diameter cannot be
replaced by diameter in Theorem~\ref{thm-main}.  In Section~\ref{sec-general}, we give the proof of Theorem~\ref{thm-main},
deferring the parts specific to the triangle-free case and to the non-triangle-free case to sections
Section~\ref{sec-trfree} and \ref{sec-3choos}.

\section{Counterexamples for diameter coloring}\label{sec-cex}

In this section, for every positive integer $\ell$ we construct
\begin{itemize}
\item a planar triangle-free graph $G_\ell$ such that any coloring of $G_\ell$ by two colors
contains a monochromatic component of diameter at least $\ell$, and
\item a planar graph $G'_\ell$ such that any coloring of $G'_\ell$ by three colors
contains a monochromatic component of diameter at least $\ell$.
\end{itemize}
This shows that it is necessary to consider the weak diameter in both cases of Theorem~\ref{thm-main}.

\subsection{The triangle-free case}

For a positive integer $k$, consider the following graphs $H_{0,k}$, $H_{1,k}$, \ldots, each with two distinct \emph{interface vertices} $u$ and $v$:
The graph $H_{0,k}$ consists just of the vertices $u$ and $v$.  For $i\ge 1$, the graph $H_{i,k}$ consists of a path $P=v_1\ldots v_k$
such that for $j=1,\ldots, k$, $v_j$ is adjancent to $u$ if $j$ is odd and to $v$ if $j$ is even,
and $k$ copies of $H_{i-1,k}$ such that the $j$-th one has interface vertices $v_j$ and $v$ if $j$ is odd and $v_j$ and $u$ if $j$ is even;
see Figure~\ref{fig-cex}.
Note that $H_{i,k}$ is planar, triangle-free, and can be drawn so that its interface vertices are incident with the outer face.
For a coloring $\varphi$ and a vertex $x$, let $r_\varphi(x)$ denote the maximum distance between $x$ and another vertex in the same
monochromatic component of $\varphi$.

\begin{figure}
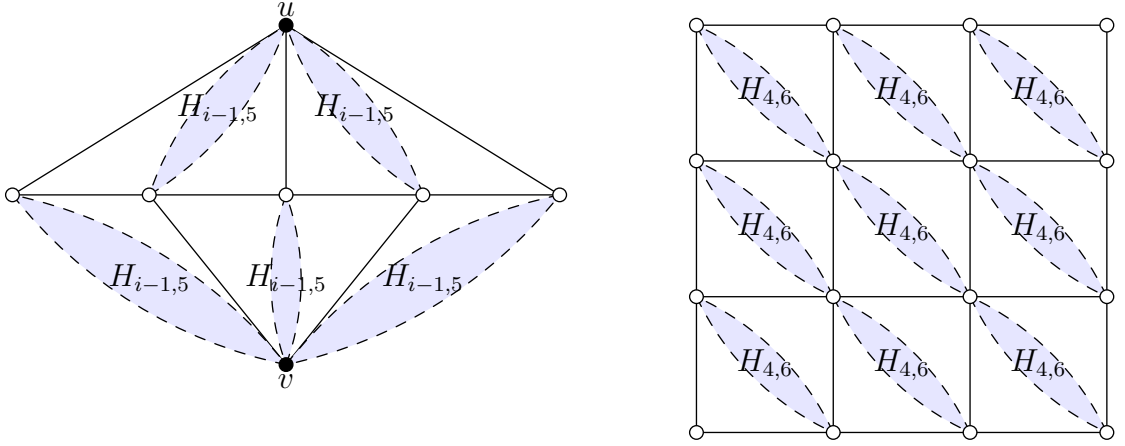

\begin{center}
\begin{asy}
v[0] = (0,0);
v[1] = (-4,-2.5);
v[2] = (-2,-2.5);
v[3] = (0,-2.5);
v[4] = (2,-2.5);
v[5] = (4,-2.5);
v[6] = (0,-5);

for (i = 1; i < 5; ++i)
  draw(v[i] -- v[i+1]);
draw (v[0]--v[1]);
draw (v[0]--v[3]);
draw (v[0]--v[5]);
draw (v[6]--v[2]);
draw (v[6]--v[4]);

for (i = 1; i <= 5; ++i)
  {
    int t = (i

    real ait = degrees(v[t] - v[i]);

    filldraw (v[i]{dir(ait - 20)}..{dir(ait+20)}v[t]{dir(180+ait-20)}..{dir(180+ait+20)}cycle, interp(blue,white, 0.9), dashed);
    label("$H_{i-1,5}$", interp(v[i],v[t],0.5));
  }

for (i = 1; i <= 5; ++i)
  vertex (v[i]);
vertex(v[0], black);
vertex(v[6], black);
label("$u$", v[0], N);
label("$v$", v[6], S);

v[10] = (6,0);
int j;
real eln = 2;
for (i = 0; i < 4; ++i)
  for (j = 0; j < 4; ++j)
    {
      if (i < 3)
        draw (v[10] + eln * (i, -j) -- v[10] + eln * (i+1, -j));
      if (j < 3)
        draw (v[10] + eln * (i, -j) -- v[10] + eln * (i, -j-1));
    
      if (i < 3 && j < 3)
        {
	  pair s = v[10] + eln * (i, -j);
	  pair t = s + eln * (1, -1);

          real ait = degrees(t - s);

          filldraw (s{dir(ait - 20)}..{dir(ait+20)}t{dir(180+ait-20)}..{dir(180+ait+20)}cycle, interp(blue,white, 0.9), dashed);
          label("$H_{4,6}$", interp(s,t,0.5));
        }
    }

for (i = 0; i < 4; ++i)
  for (j = 0; j < 4; ++j)
    vertex (v[10] + eln * (i, -j));
\end{asy}
\end{center}
\caption{The graphs $H_{i,5}$ and $G_3$.}\label{fig-cex}
\end{figure}

\begin{lemma}\label{lemma-nomo}
For all integers $i\ge 0$ and $k\ge 1$, if a coloring $\varphi$ assigns colors $1$ and $2$ to vertices of the graph $H_{i,k}$
and both interface vertices $u$ and $v$ receive color $1$, then
\begin{itemize}
\item $\varphi$ contains a component of color $2$ of diameter at least $k-1$, or 
\item $u$ and $v$ are in the same monochromatic component of $\varphi$, or
\item $r_\varphi(u)+r_\varphi(v)\ge i$.
\end{itemize}
\end{lemma}
\begin{proof}
We prove the claim by induction on $i$.  The case $i=0$ is trivial, and thus we can assume $i\ge 1$.
Let $P=v_1v_2\ldots v_k$ be the path in $H_{i,k}$ from the definition.
Suppose that $\varphi$ does not contain a component of color $2$ of diameter at least $k-1$;
then the whole path $P$ cannot be colored by color $2$, and thus $\varphi(v_j)=1$ for some $j\in\{1,\ldots,k\}$.  By symmetry, we can assume that $j$ is odd.
Consider the copy of $H_{i-1,k}$ with interface vertices $v$ and $v_j$, and let $\varphi'$ be the restriction of of $\varphi$ to this copy.
Then $\varphi'$ also cannot contain a component of color $2$ of diameter at least $k-1$.

Moreover, suppose that $u$ and $v$ are in different monochromatic components of $\varphi$; since $uv_j$ is an edge and $\varphi(u)=\varphi(v_j)=1$,
this implies that $v$ and $v_j$ are not in the same monochromatic component of $\varphi'$.

By the induction hypothesis, it follows that $r_{\varphi'}(v_j)+r_{\varphi'}(v)\ge i-1$.  Since $u$ and $v$ are in different monochromatic components,
a monochromatic path from $u$ cannot pass through $v$, and thus $r_\varphi(u)\ge r_{\varphi'}(v_j)+1$.  Similarly, $r_\varphi(v)\ge r_{\varphi'}(v)$.
Therefoore, $r_\varphi(u)+r_\varphi(v)\ge (r_{\varphi'}(v_j)+1)+r_{\varphi'}(v)\ge i$.
\end{proof}

Let $W_\ell$ be the $(\ell+1)\times (\ell+1)$ grid with diagonals added to the 4-faces.
Let $G_\ell$ be the graph obtained from $W_\ell$ by replacing each diagonal $uv$ by
a copy of $H_{\ell+1,2\ell}$ with interface vertices $u$ and $v$, see the right part of Figure~\ref{fig-cex}.
Consider any 2-coloring of $G_\ell$.  By the HEX lemma, the corresponding $2$-coloring of $W_\ell$
contains a monochromatic path $Q$ (say in color $1$) joining the opposite sides of the grid.  If both ends
of $Q$ belong to the same monochromatic component of $\varphi$ on $G_\ell$, then this component has diameter
greater than $\ell$.  Otherwise, there exists an edge $uv\in E(Q)$ such that $u$ and $v$ belong to different
monochromatic components of $\varphi$, and thus also to different monochromatic components of $\varphi$ restricted
to the copy of $H_{\ell+1,2\ell}$ with interface vertices $u$ and $v$.  By Lemma~\ref{lemma-nomo}, this implies
that either one of the monochromatic components of $u$ and $v$ has diameter at least $\ell$, or $\varphi$ contains
a component of color $2$ of diameter at least $\ell$.

\subsection{The non-triangle-free case}

\begin{figure}
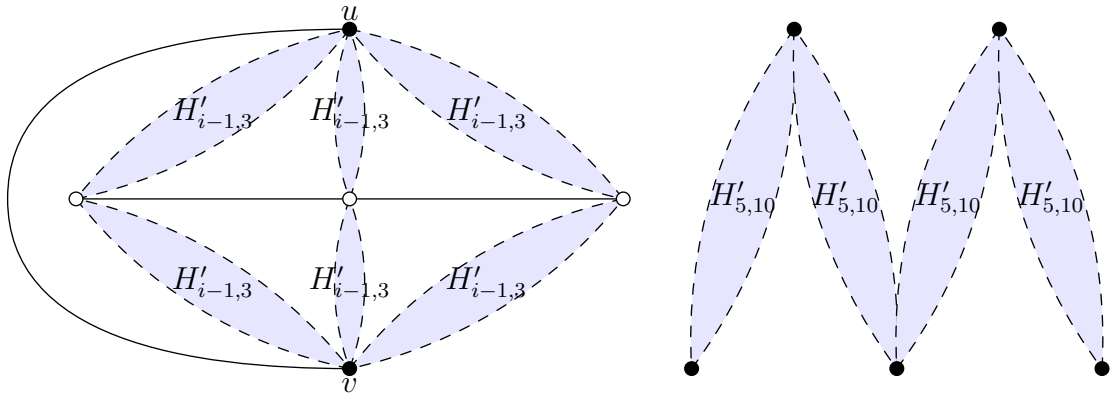

\begin{center}
\begin{asy}
v[0] = (0,0);
v[1] = (-4,-2.5);
v[2] = (0,-2.5);
v[3] = (4,-2.5);
v[4] = (0,-5);

for (i = 1; i < 3; ++i)
  draw(v[i] -- v[i+1]);

draw (v[0]{dir(180)} .. (-5,-2.5) .. {dir(0)}v[4]);

int t;
for (i = 1; i <= 3; ++i)
  for (t = 0; t <= 4; t+=4)
    {
      real ait = degrees(v[t] - v[i]);

      filldraw (v[i]{dir(ait - 20)}..{dir(ait+20)}v[t]{dir(180+ait-20)}..{dir(180+ait+20)}cycle, interp(blue,white, 0.9), dashed);
      label("$H'_{i-1,3}$", interp(v[i],v[t],0.5));
    }

for (i = 1; i <= 3; ++i)
  vertex (v[i]);
vertex(v[0], black);
vertex(v[4], black);
label("$u$", v[0], N);
label("$v$", v[4], S);

v[10] = (5,-5);
v[11] = (6.5,0);
v[12] = (8,-5);
v[13] = (9.5,0);
v[14] = (11,-5);

for (i = 0; i < 4; ++i)
  {
    int s = 10+i;
    int t = 11+i;

    real ait = degrees(v[t] - v[s]);

    filldraw (v[s]{dir(ait - 20)}..{dir(ait+20)}v[t]{dir(180+ait-20)}..{dir(180+ait+20)}cycle, interp(blue,white, 0.9), dashed);
    label("$H'_{5,10}$", interp(v[s],v[t],0.5));
  }

for (i = 0; i < 5; ++i)
  vertex(v[10+i], black);
\end{asy}

\end{center}
\caption{The graphs $H'_{i,3}$ and $G'_4$.}\label{fig-cex2}
\end{figure}

For the case of general planar graphs colored by three colors, we use a similar construction.  The graph $H'_{i,k}$ whose recursive
construction is depicted in Figure~\ref{fig-cex2} has the following property: Suppose that the interface vertices $u$ and $v$
receive colors $1$ and $2$, respectively.  Then
\begin{itemize}
\item the $k$-vertex path is colored by $3$ (resulting in a monochromatic component
of diameter at least $k-1$), or
\item a vertex $x$ of the path has color $1$, there exists a copy of $H'_{i-1,k}$ with interface vertices
$x$ and $v$ of colors $1$ and $2$, and the monochromatic component of $x$ additionally contains the edge $xu$, or
\item a symmetric situation with a vertex $x$ of the path receiving color $2$.
\end{itemize}
As in the previous case, this implies that any coloring of $H'_{\ell+1,2\ell}$
by three colors where the interface vertices receive a different color contains a monochromatic component of diameter at least $\ell$.
The graph $G'_{\ell}$ is then obtained by concatenating $\ell$ copies of this graph,
as depicted on the right side of Figure~\ref{fig-cex2}.

\section{The proof of Theorem~\ref{thm-main}}\label{sec-general}

In this section we present the common parts of the proofs of Theorem~\ref{thm-main} (i) and (ii). 
We use $c$ to denote the number of colors
in each list and $t$ the lower bound on the girth of the considered graph, where 
 $c = 3$ and $t=3$ for the proof of Theorem~\ref{thm-main} (i), while  $c = 2$ and $t=4$ for Theorem~\ref{thm-main} (ii).
 Observe that $t=\tfrac{2c}{c-1}$ in both cases.

The starting point of our proof is a standard island argument.
A non-empty set $I\subseteq V(G)$ is a \emph{$c$-island} if every vertex in $I$ has less than $c$ neighbors outside of $I$.
For real numbers $a$ and $b$, we say that a graph $G$ is \emph{$(a,b)$-sparse} if $|E(G)|\le a|V(G)|+b$,
and \emph{hereditarily $(a,b)$-sparse} if every induced subgraph of $G$ is $(a,b)$-sparse.
In~\cite{islands} we proved the following claim.
\begin{lemma}\label{lemma-islands}
For all positive integers $c$ and $b$, any real number $\varepsilon>0$,
and every surface $\Sigma$, there exists a positive integer $s$ such that the following claim holds:
Every $(c-\varepsilon,b)$-sparse graph $G$ with $V(G)\neq\emptyset$ drawn on $\Sigma$ contains a $c$-island of size at most $s$.
\end{lemma}

The presence of $c$-islands can be used to obtain clustered colorings.

\begin{corollary}\label{cor-colin}
For all positive integers $c$ and $b$, any real number $\varepsilon>0$,
and every surface $\Sigma$, the class of hereditarily $(c-\varepsilon,b)$-sparse
graphs drawn on $\Sigma$ has clustered choosability at most $c$.
\end{corollary}
\begin{proof}
Let $s$ be the constant from Lemma~\ref{lemma-islands}.  We show that every hereditarily $(c-\varepsilon,b)$-sparse
graph $G$ drawn on $\Sigma$ has an $L$-coloring with clustering at most $s$ for any assignment $L$ of lists of size $c$.
We prove the claim by induction on the number of vertices of $G$.  The claim is trivial if $G$ has no vertices.
Otherwise, by Lemma~\ref{lemma-islands}, $G$ contains a $c$-island $I$ of size at most $s$.
By the induction hypothesis, $G-I$ has an $L$-coloring with clustering at most $s$.
We color $I$ so that each vertex $v\in I$ chooses a color from $L(v)$ different from the colors of
its neighbor outside of $I$.  This ensures that any newly arising monochromatic components are contained in $I$,
and thus they have size at most $s$.
\end{proof}

Corollary~\ref{cor-colin} ``just barely'' does not apply in the setting of Theorem~\ref{thm-main}. Indeed,
the generalized Euler's formula implies that an $n$-vertex graph of girth at least $t=\tfrac{2c}{c-1}$ drawn on a surface of
Euler genus $g$ has at most $\tfrac{t}{t-2}n+O(g)=cn+O(g)$ edges. Hence, if we are able to make the graph just a bit
sparser, we can finish the argument using Corollary~\ref{cor-colin}.  To this end, let us introduce the notion of \emph{sparsifiers}.

A \emph{multiassignment} for a graph $S$ is an assignment of multisets to vertices of $S$.
For a multiassignment $B$ to vertices of a graph $S$, we say that a coloring of $S$ is \emph{$B$-opaque} if for each color $a$,
each connected subgraph of $S$ of color $a$ contains at most one vertex $v$ such that $a\in B(v)$ and if there is such a vertex $v$,
then $a$ appears in $B(v)$ with multiplicity one.  The motivation for this definition
is as follows: If $S$ is an induced subgraph of a colored graph $G$ and $B(u)$ consists of colors that appear on the neighbors of $u$
in $V(G)\setminus V(S)$, then $B$-opacity implies that no two monochromatic components of $G-V(S)$ are contained in the same monochromatic
component of $G$.

A \emph{$c$-sparsifier} is a pair $(S,\gamma)$, where $S$ is a connected graph and $\gamma:V(S)\to\mathbb{Z}_0^+$ assigns an integer $\gamma(v)\ge \deg v$
to each vertex $v\in V(S)$, with the following property:  For any assignment $L$ of lists of size $c$ to vertices of $S$
and a multiassignment $B$ of lists to vertices of $S$ such that $|B(v)|\le \gamma(v)-\deg v$ for each $v\in V(S)$, there exists
a $B$-opaque $L$-coloring of $S$.  The \emph{size} of the sparsifier is $|V(S)|$.  An \emph{appearance} of a $c$-sparsifier $(S,\gamma)$ in a graph $G$ drawn on a surface
is an injective function $h:V(S)\to V(G)$ such that
\begin{itemize}
\item for $u,v\in V(S)$, we have $uv\in E(S)$ if and only if $h(u)h(v)\in E(G)$ (i.e., $h$ shows $S$ is an induced subgraph of $G$),
\item for $u\in V(S)$, we have $\deg_G h(u)\le \gamma(u)$, and
\item for $u\in V(S)$, every face of $G$ incident with $h(u)$ is bounded by a cycle of length $t=\tfrac{2c}{c-1}$.
\end{itemize}
We write $G-h$ for the graph obtained from $G$ by deleting all vertices in the image of $h$.
Appearances $h_1$ and $h_2$ of $c$-sparsifiers $(S_1,\gamma_1)$ and $(S_2,\gamma_2)$ are \emph{independent}
if $h_1(u)\neq h_2(v)$ and $h_1(u)h_2(v)\not\in E(G)$ for every $u\in V(S_1)$ and $v\in V(S_2)$.

The definition of a sparsifier and its appearance is motivated by the following properties.
\begin{lemma}\label{lemma-delspars}
Let $G$ be a graph drawn on a surface, let $c\ge 2$, $p\ge 1$ and $\ell\ge 0$ be integers,
let $L$ be an assignment of lists of size $c$ to vertices of $G$, and let $h_1$, \ldots, $h_m$ be pairwise-independent
appearances of $c$-sparsifiers of size at most $p$ in $G$.  If $G-\{h_1,\ldots, h_m\}$ has a weak diameter-$\ell$ $L$-coloring $\varphi$, then
$G$ has a weak diameter-$(\ell+2p)$ $L$-coloring.
\end{lemma}
\begin{proof}
For $i\in \{1,\ldots,m\}$, we extend $\varphi$ to the image of $h_i$ as follows.  Let $(S_i,\gamma_i)$ be the $c$-sparsifier with appearance $h_i$.
For $u\in V(S_i)$, let $L_i(u)=L(h_i(u))$ and $B_i(u)=\{\varphi(v):vh_i(u)\in E(G),v\in V(G-\{h_1,\ldots, h_m\})\}$.  Since $\deg_G h_i(u)\le \gamma_i(u)$,
we have $|B_i(u)|\le \gamma_i(u)-\deg_S u$.  By the definition of a $c$-sparsifier, there exists a $B_i$-opaque $L_i$-coloring $\psi_i$
of $S_i$, and for each $u\in V(S_i)$, we define $\varphi(h_i(u))=\psi_i(u)$.

Since $\psi_i$ is $B_i$-opaque for each $i$, each monochromatic component of $G$ in the coloring $\varphi$ is either contained
in the image of $h_i$ for some $i$, or it is obtained from a monochromatic component of $G-\{h_1,\ldots, h_m\}$ by adding disjoint
non-adjacent connected subgraphs with at most $p$ vertices.  We conclude that each monochromatic component of $G$ has weak diameter at most $\ell+2p$.
\end{proof}

A system $h_1$, \ldots, $h_m$ of pairwise-independent appearances of $c$-sparsifiers of size at most $p$ in a graph $G$ drawn on a surface is \emph{maximal}
if there does not exist an appearance of a $c$-sparsifier of size at most $p$ in $G$ independent of $h_1$, \ldots, $h_m$.
A graph $G$ is \emph{$(c,p)$-sparsifier-free} if no $c$-sparsifier of size at most $p$ has an appearance in $G$.

\begin{lemma}\label{lemma-nonewspars}
Let $G$ be a graph drawn on a surface, let $c\ge 2$ and $p\ge 1$ be integers,
and let $h_1$, \ldots, $h_m$ be pairwise-independent appearances of $c$-sparsifiers of size at most $p$ in $G$.
Let $t=\tfrac{2c}{c-1}$.  If the system $h_1$, \ldots, $h_m$ is maximal, $|V(G)|>t$ and $G$ does not contain any separating
cycle of length $t$, then every induced subgraph $G'$ of $G-\{h_1,\ldots,h_m\}$ is $(c,p)$-sparsifier-free.
\end{lemma}
\begin{proof}
Suppose for a contradiction that $h$ is an appearance of a $c$-sparsifier $(S,\gamma)$ of size at most $p$ in $G'$.
By the last condition in the definition of an appearance, all faces incident with the $h$-images of vertices of $S$
are bouded by $t$-cycles.  Since $G$ does not contain separating $t$-cycles, $|V(G)|>t$, and $G'$ is an induced subgraph
of $G$, these faces are also faces of $G$.  In particular, all the vertices in the image of $h$ have the same degree in $G$
as in $G'$.  Hence, $h$ is also an appearance of $(S,\gamma)$ in $G$ independent from $h_1$, \ldots, $h_m$, contradicting the
maximality of the system.
\end{proof}

Finally, we will need the following lemma, whose proof is specific to the cases $c\in\{2,3\}$ and is given in Sections~\ref{sec-trfree}
and \ref{sec-3choos}.
\begin{lemma}\label{lemma-spaspar}
For $c\in\{2,3\}$, there exists a constant $\varepsilon_c>0$ such that the following claim holds.
Let $G$ be a graph of minimum degree at least $c$ and girth at least $t=\tfrac{2c}{c-1}$ drawn on a surface of Euler genus $g$
with no non-contractible cycles of length at most four.  Suppose that $G$ is $(c,4)$-sparsifier-free
and does not contain separating cycles of length $t$.  Then $G$ is $(c-\varepsilon_c,10(g+3))$-sparse.
\end{lemma}

Let us now combine these claims.

\begin{corollary}\label{cor-nocut}
For $c \in \{2,3\}$ and every surface $\Sigma$, the class $\GG_{c,\Sigma}$ of graphs of girth at least $t=\tfrac{2c}{c-1}$ drawn on $\Sigma$
with no non-contractible cycles of length at most four and no separating cycles of length $t$ has weak diameter choosability at most $c$.
\end{corollary}
\begin{proof}
Let $\varepsilon_c>0$ be the constant from Lemma~\ref{lemma-spaspar} and let $g$ be the Euler genus of $\Sigma$.
By Corollary~\ref{cor-colin}, there exists $s$ such that every hereditarily $(c-\varepsilon_c,10(g+3))$-sparse
graph drawn on $\Sigma$ has a coloring with clustering at most $s$ from any assignment of lists of size $c$.

Consider a graph $G\in \GG_{c,\Sigma}$ and an assignment $L$ of lists of size $c$ to vertices of $G$.
Let $h_1$, \ldots, $h_m$ be a maximal system of pairwise-independent $c$-sparsifiers of size at most $4$ in $G$.
Let $G_0=G-\{h_1,\ldots,h_m\}$.  We claim that $G_0$ is hereditarily $(c-\varepsilon_c,10(g+3))$-sparse.
Hence, we need to prove that every induced subgraph $G'$ of $G_0$ is $(c-\varepsilon_c,10(g+3))$-sparse.
We prove the claim by induction on $|V(G')|$.  If $|V(G')|\le t\le 4$, then the claim is trivial since $10(g+3)\ge 10\ge |E(G')|$.
In particular, we can assume that $|V(G)|>t$, and Lemma~\ref{lemma-nonewspars} implies $G'$ is $(c,4)$-sparsifier-free.
If a vertex $v\in V(G')$ has degree at most $c-1$, then
$|E(G')|\le (c-1)+|E(G'-v)|\le (c-1)+(c-\varepsilon_c)|V(G'-v)|+10(g+3)\le (c-\varepsilon_c)|V(G')|+10(g+3)$
by the induction hypothesis.  On the other hand, if $G'$ has minimum degree at least $c$, then $G'$ is $(c-\varepsilon_c,10(g+3))$-sparse
by Lemma~\ref{lemma-spaspar}.

By Corollary~\ref{cor-colin}, $G_0$ has an $L$-coloring $\varphi$ with clustering at most $s$.  Then $\varphi$ is also
a weak diameter-$(s-1)$ coloring.  By Lemma~\ref{lemma-delspars}, $G$ has a weak diameter-$(s+7)$ $L$-coloring.
\end{proof}

\begin{figure}
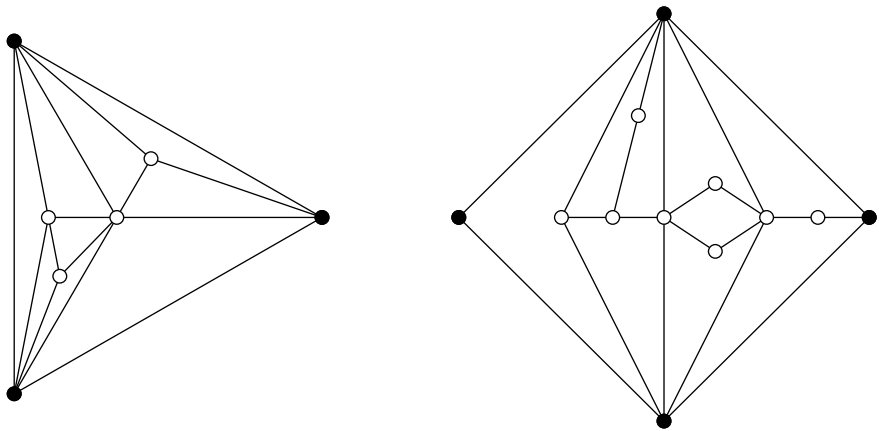

\begin{center}
\begin{asy}
v[0] = (0,0);
for (i = 1; i <= 3; ++i)
  v[i] = v[0] + 3 * dir(120i);
v[4] = (v[0]+v[1]+v[2])/3;
v[5] = (v[0]+v[1]+v[3])/3;
v[6] = (v[0]+v[2]+v[4])/3;

draw(v[1]--v[2]--v[3]--cycle);
draw(v[0]--v[1]);
draw(v[0]--v[2]);
draw(v[0]--v[3]);

draw(v[4]--v[0]);
draw(v[4]--v[1]);
draw(v[4]--v[2]);

draw(v[5]--v[0]);
draw(v[5]--v[1]);
draw(v[5]--v[3]);

draw(v[6]--v[0]);
draw(v[6]--v[2]);
draw(v[6]--v[4]);

for (i = 0; i <= 6; ++i)
  vertex (v[i]);
for (i = 1; i <= 3; ++i)
  vertex (v[i],black);

v[0] = (8,0);
for (i = 1; i <= 4; ++i)
  v[i] = v[0] + 3 * dir(90i);
v[5] = (6.5,0);
v[6] = (9.5,0);
v[7] = (v[0] + v[5]) / 2;
v[8] = (v[7] + v[1]) / 2;
v[9] = (v[6] + v[0]) / 2 + (0,0.5);
v[10] = (v[6] + v[0]) / 2 + (0,-0.5);
v[11] = (v[6] + v[4]) / 2;

draw(v[1]--v[2]--v[3]--v[4]--cycle);
draw(v[1]--v[0]--v[3]);
draw(v[1]--v[5]--v[3]);
draw(v[1]--v[6]--v[3]);
draw(v[0]--v[7]--v[5]);
draw(v[7]--v[8]--v[1]);
draw(v[0]--v[9]--v[6]);
draw(v[0]--v[10]--v[6]);
draw(v[4]--v[11]--v[6]);

for (i = 0; i <= 11; ++i)
  vertex (v[i]);
for (i = 1; i <= 4; ++i)
  vertex (v[i],black);
\end{asy}
\end{center}
\caption{A 3-stack and a 4-stack.}\label{fig-stacks}
\end{figure}

Next, we need to take care of separating $t$-cycles.  These are generally dealt with using standard precoloring arguments,
but the cases where a vertex has $c$ precolored neighbors turn out to be somewhat problematic and require us to
handle the following special case separately.  For a cycle $K$ in a plane graph $G$, let $G_K$ denote the subgraph of $G$ drawn in the closed disk bounded
by $K$.  The \emph{$3$-base} is the plane drawing of $K_4$ and the \emph{$4$-base}
is the plane drawing of $K_{2,3}$.  For $t\in \{3,4\}$, a finite plane graph $G$ is a \emph{$t$-stack} if it is either
a cycle of length $t$, or if there exists a $t$-base $H\subseteq G$ such that the outer face of $G$ is equal to the
outer face of $H$ and for each internal face of $H$ bounded by a $t$-cycle $K$, the graph $G_K$ is a $t$-stack.
See Figure~\ref{fig-stacks} for an example of a $3$-stack and a $4$-stack.  Let $C$ be the cycle bounding the outer
face of a $t$-stack $G$, let $\psi$ be a coloring of $C$ and let $\varphi$ be a coloring of $G$ that extends $\psi$.
We say that $\varphi$ is \emph{$\psi$-opaque} if no monochromatic component of $\varphi$ on $G$
contains vertices belonging to two distinct monochromatic components of $\psi$ on $C$.
The proof of the following lemma is specific to the cases $c\in\{2,3\}$ and is given in Sections~\ref{sec-trfree}
and \ref{sec-3choos}.

\begin{lemma}\label{lemma-stacks}
For $c\in\{2,3\}$, let $t=\tfrac{2c}{c-1}$.  Let $G$ be a $t$-stack and let $L$ be an assignment of lists of size $c$
to vertices of $G$.  Then every $L$-coloring $\psi$ of the cycle $C$ bounding the outer face of $G$
extends to a weak diameter-$4$ $\psi$-opaque $L$-coloring $\varphi$ of $G$.
\end{lemma}

A cycle $C$ in a graph $G$ is \emph{$c$-solitary} if every vertex $v\in V(G)\setminus V(C)$
has fewer than $c$ neighbors in $C$.  Given a coloring $\psi$ of $C$, we say that a coloring
$\varphi$ of $G$ \emph{properly extends} $\psi$ if the restriction of $\varphi$ to $C$
is equal to $\psi$ and $\psi(u)\neq \varphi(v)$ for every $uv\in E(G)$ such that $u\in V(C)$
and $v\not\in V(C)$.  For a graph $G$ drawn on a surface of non-zero Euler genus and a contractible cycle $K$ in $G$,
let $G_K$ denote the subgraph of $G$ drawn in the unique closed disk in the surface bounded by $K$.

\begin{lemma}\label{lemma-ext}
For $c\in\{2,3\}$  and  every surface $\Sigma$,
there exists a positive integer $\ell$ such that the following claim holds.
Let $G$ be a graph of girth at least $t=\tfrac{2c}{c-1}$ drawn on $\Sigma$ without non-contractible cycles of length at most four
and let $L$ be an assignment of lists of size $c$ to vertices of $G$.
Suppose that either $C$ is an empty graph, or $\Sigma$ is the plane and $C$ is a cycle of length $t$ bounding the outer face of $G$,
and let $\psi$ be an $L$-coloring of $C$.  If $C$ is $c$-solitary, then $\psi$ properly extends to a weak diameter-$\ell$ $L$-coloring of $G$.
\end{lemma}
\begin{proof}
Let $\Sigma_0$ be the sphere.  By Corollary~\ref{cor-nocut}, there exists $\ell'$ such that every graph from $\GG_{c,\Sigma}\cup \GG_{c,\Sigma_0}$
has a weak diameter-$\ell'$ coloring from any assignment of lists of size $c$.  Let $\ell=2\ell'+22$.

We prove the claim by induction on the number of vertices of $G$.
Suppose first that there exists a separating $t$-cycle $K$ in $G$ (necessarily contractible, since $t\le 4$) such that $K$ is $c$-solitary in $G_K$.
Let $G_1$ be the graph obtained from $G$ by deleting the vertices and edges drawn in the open disk bounded by $K$.
By the induction hypothesis, $\psi$ properly extends to a weak diameter-$\ell$ $L$-coloring $\varphi_1$ of $G_1$.
Using the induction hypothesis again, the restriction of $\varphi_1$ to $K$ properly extends to
a weak diameter-$\ell$ $L$-coloring $\varphi_2$ of $G_K$ ($G_K$ is drawn in the plane rather than in $\Sigma$ when
$\Sigma$ has positive genus, but this is not a problem, as we included the genus-0 case in the choice of $\ell'$).  Since $\varphi_2$ properly extends the restriction of $\varphi_1$,
each monochromatic component of $G$ in the $L$-coloring $\varphi_1\cup \varphi_2$ is contained in $G_1$ or $G_K$,
and thus has weak diameter at most $\ell$.

Hence, we can assume there is no such separating $t$-cycle.  Observe that this implies that for each separating $t$-cycle, the graph $G_K$
is a $t$-stack.  Let $K_1$, \ldots, $K_m$ be separating $t$-cycles in $G$ such that the open disks bounded by them are inclusionwise-maximal,
and observe that these open disks are disjoint.  Let $G'$ be the graph obtained from $G$ by deleting vertices and edges
drawn in these disks.  Then $G'$ has no separating $t$-cycles, and thus $G'\in \GG_{c,\Sigma}$.
By Corollary~\ref{cor-nocut}, $G'-V(C)$ has a weak diameter-$\ell'$ $L$-coloring $\varphi'$.
For $i\in\{1,\ldots,m\}$, Lemma~\ref{lemma-stacks} implies the restriction $\psi_i$ of $\varphi'\cup \psi$ to $K_i$
extends to a weak diameter-$4$ $\psi_i$-opaque $L$-coloring $\varphi_i$ of $G_{K_i}$.
Then $\varphi''=\varphi'\cup \varphi_1\cup\ldots\cup \varphi_m$ is a weak diameter-$(\ell'+8)$ $L$-coloring of $G-V(C)$.

Let $\varphi$ be the $L$-coloring that matches $\psi$ on $C$, $\varphi(v)\in L(v)$ is chosen as an arbitrary color different
from the colors of the neighbors of $v$ in $C$ for every vertex $v\in V(G)\setminus V(C)$
with at least one neighbor in $C$ (this is possible, since $C$ is $c$-solitary), and $\varphi(v)=\varphi''(v)$ for each vertex $v$
at distance at least two from $C$.
Every monochromatic component in $\varphi$ not contained in $C$ is obtained from a disjoint union of connected monochromatic subgraphs in $\varphi''$
by adding neighbors of vertices of $C$, and thus the distance between any two vertices of the resulting monochromatic component
is at most $2(\ell'+8)+6=\ell$.
\end{proof}

To finish the proof, we need to deal with non-contractible cycles of length at most 4.

\begin{proof}[Proof of Theorem~\ref{thm-main}]
For $c\in\{2,3\}$, let $t=\tfrac{2c}{c-1}$.  For a non-negative integer $g$, let $\ell'_g$ be the
maximum of the constants $\ell$ from Lemma~\ref{lemma-ext} over all surfaces of Euler genus at most $g$,
and let us define $\ell_0=\ell'_0$ and $\ell_g=\max(\ell'_g,2\ell_{g-1}+4)$.
We prove by induction on $g$ that any graph $G$ of girth at least $t$ drawn on a surface of Euler genus at
most $g$ has a weak diameter-$\ell_g$ $L$-coloring from any assignment $L$ of lists of size $c$.

If $g=0$, then $G$ does not contain any non-contractible cycles, and thus the claim follows from Lemma~\ref{lemma-ext}
(with $C$ being an empty graph, and considering the drawing of $G$ in the plane instead of on the sphere).
Hence, suppose that $g>0$.  If $G$ does not contain any non-contractible cycle of length at most $4$, then
the claim again follows from Lemma~\ref{lemma-ext}.  Hence, suppose $K$ is a non-contractible cycle of length at most $4$ in $G$.
Then each component of the graph $G-V(K)$ can be drawn on a surface of Euler genus at most $g-1$, and by the induction
hypothesis, $G-V(K)$ has a weak diameter-$\ell_{g-1}$ $L$-coloring.  We extend this $L$-coloring to $G$ by choosing the colors
of vertices of $K$ from their lists arbitrarily; each monochromatic component of the resulting $L$-coloring has
weak diameter at most $2\ell_{g-1}+4\le \ell_g$, as required.
\end{proof}

\section{The triangle-free case}\label{sec-trfree}

Let us now provide the proofs of Lemmas~\ref{lemma-spaspar} and~\ref{lemma-stacks} in the case $c=2$.
Let $S_1$ be a single vertex and $\gamma_1$ the function assigning to this vertex the value $3$,
and let $S_2$ be the 4-cycle and $\gamma_2$ the function assigning to all its vertices the value $4$.

\begin{lemma}\label{lemma-spars}
Both $(S_1,\gamma_1)$ and $(S_2,\gamma_2)$ are $2$-sparsifiers of size at most 4.
\end{lemma}
\begin{proof}
Consider $i\in \{1,2\}$, let $L$ be an assignment of lists of size $2$ to vertices of $S_i$, and let $B$
be a multiassigment of a list of size $3$ in case $i=1$ and of lists of size $2$ in case $i=2$.
In the case $i=1$, choose $\varphi(v)\in L(v)$ to be different from the color that appears in $B(v)$ twice (if any).
Clearly, $\varphi$ is $B$-opaque.

In the case $i=2$, we choose the $L$-coloring of the 4-cycle $S_2$ as follows.
Let $R$ be the set of vertices $v\in V(S_2)$ such that $B(v)$ contains some color $a$ with multiplicity two;
for such a vertex, set $L'(v)=L(v)\setminus\{a\}$.  For any vertex $v\in V(S_2)\setminus R$, set $L'(v)=L(v)$.
Orient the cycle $S_2$ arbitrarily, and let $S'$ be the graph obtained from $S_2$ by, for each vertex $v\in R$,
deleting the edge that follows it in $S_2$ in this orientation.  Note that $|L'(v)|\ge \deg_{S'} v$ for each
$v\in V(S_2)$, and that either $S'$ is a $4$-cycle, or the first vertex $u$ of each component of $S'$ according
to the orientation of the 4-cycle satisfies $|L'(u)|>\deg_{S'} u$.  Consequently, $S'$ has a proper $L'$-coloring
$\varphi$.  We claim that $\varphi$ is $B$-opaque.  Indeed, consider distinct vertices $v_1,v_2\in V(S_2)$ such that
$\varphi(v_1)=\varphi(v_2)=a\in B(v_1)\cap B(v_2)$.  Clearly, $v_1,v_2\in V(S_2)\setminus R$, and thus
the vertices following $v_1$ and $v_2$ in $S_2$ have colors different from $a$.  Hence, $v_1$ and $v_2$ are not
in the same monochromatic component.
\end{proof}

\begin{proof}[Proof of Lemma~\ref{lemma-spaspar} in the case $c=2$]
Let $\varepsilon_2=1/3000$.  We can assume $|V(G)|>5$, as otherwise the claim holds trivially.
Since $G$ is triangle-free and has minimum degree at least two, every face of $G$ has length at least four.
Let $\beta=\sum_{f\in F} (|f|-4)$, where the sum is over all faces.  By the generalized Euler's formula,
we have $|E(G)|\le |V(G)|+|F|+g-2$, and since $2|E(G)|=\sum_{f\in F} |f|=4|F|+\beta$,
we conclude that
$$|E(G)|<2|V(G)|-\frac{\beta}{2}+2g.$$
Since $G$ is simple, does not have separating or non-contractible 4-cycles, and $|V(G)|>5$, every vertex of degree two is incident
with a face of length at least five.  Since $G$ is $(2,4)$-sparsifier-free, $(S_1,\gamma_1)$ has no appearance in $G$,
and thus each vertex of degree three is also incident with a face of length at least five.
Hence, the number $n_3$ of vertices of degree at most three is at most $5\beta$.

Let us give each vertex of degree at most three the charge $1$, any vertex of degree $d\ge 4$ charge $d-4$,
and any face $f$ the charge $|f|-4$.  By the generalized Euler's formula, the sum of charges is
at most
\begin{align*}
\Bigl(\sum_{v\in V(G)} (\deg v -4)\Bigr)+3n_3+\sum_{f\in F} (|f|-4)&=4(|E(G)|-|V(G)|-|F|)+3n_3\\
&\le 4g-8+3n_3<15\beta+4g.
\end{align*}
Each vertex $v$ of degree $d\neq 4$ now sends $1/11$ to each 
adjacent vertex and each vertex opposite to $v$ over a $4$-face; this still leaves $v$ with at least
$$\max(d-4,1)-2d/11\ge \max(9d/11-4,1-2d/11)\ge 1/11$$
units of charge.  Each face $|f|$ of length at least five sends $1/11$ to each incident vertex,
still keeping $|f|-4-|f|/11>0$ units of charge.
Afterwards, each vertex $v$ of degree four which received charge sends $1/99$ to each adjacent vertex and each vertex opposite to $v$
over a 4-face; note that $v$ keeps at least $1/11-8/99=1/99$ units of charge.
Since each face has non-negative final charge and the total amount of charge did not change, we conclude that the sum of
the final charges of vertices is less than $15\beta+4g$.  Note that each vertex has non-negative final charge, and vertices of
degree other than four have final charge at least $1/11$.

We claim that vertices of degree four have charge at least $1/99$.  Consider for a contradiction a vertex $v$
of degree four with smaller final charge.  All incident faces must be 4-faces, only incident
with vertices of degree four, and the faces incident those must also have length
four.  This is not possible, since $(S_2,\gamma_2)$ does not have an appearance in $G$.

Since every vertex has final charge at least $1/99$, we have $|V(G)|/99\le 15\beta+4g$,
and thus $\beta\ge |V(G)|/1500-g$.
Consequently, $|E(G)|<2|V(G)|-\beta/2+2g<(2-1/3000)|V(G)|+3g$.
\end{proof}

We finish this section by proving a strengthening Lemma~\ref{lemma-stacks}  for $c=2$. Introducing this strengthening requires the following additional definitions.
Let $C$ be the 4-cycle bounding the outer face of a 4-stack $G$, and  let $\varphi$ be a coloring of $G$. We say that a monochromatic component $Q$ of $\varphi$ is 
\emph{$C$-transversal} if $V(Q) \cap V(C) \neq \emptyset$ and $V(Q)\setminus V(C) 
\neq 
\emptyset$. We say that $\varphi$ is \emph{$v$-compliant} for some $v \in V(C)$ if either  no monochromatic component of $\varphi$ is $C$-transversal, or there exists a unique such component $Q$ and the following conditions hold
\begin{description}
	\item[(C1)] $v \in V(Q)$,
	\item[(C2)] $\varphi(v) \neq \varphi(v')$, where $v'$ is the unique non-neighbor of $v$ on $C$,
	\item[(C3)] every vertex in $V(Q)\setminus V(C)$ has a neighbor in $V(C)\setminus V(Q)$.
\end{description}
We say that a vertex $v \in V(C)$ is \emph{$G$-active} if every vertex in $V(G) \setminus V(C)$ with two neighbors on $C$ is adjacent to $v$.

\begin{lemma}\label{lemma-stacks2}
Let $G$ be a 4-stack with the outer face bounded by a 4-cycle $C$, let $L$ be an assignment of lists of size two
to vertices of $G$, let $\psi$ be an $L$-coloring of $C$, and let $v \in V(C)$ be $G$-active.
Then $\psi$ extends to a weak diameter-$4$ $\psi$-opaque $v$-compliant $L$-coloring $\varphi$
of $G$.
\end{lemma}
\begin{proof}
We prove the lemma by induction on $|V(G)|$.  The basic case $G=C$ is trivial.  Hence, we can assume $G\neq C$.
Let $v'$ is the unique vertex of $C$ non-adjacent to $v$, and let $X =
\{x_1, \ldots, x_{m+1}\}$ be the set of all common neighbors of $v$ and $v'$ in $G$, numbered so that for  every $i \in \{1,\ldots,m\}$ the cycle  $C_i=x_{i}vx_{i+1}v'$ does not contain any vertices of $X$ in its interior. In particular, we have $x_1, x_{m+1} \in V(C)$. Let $G_i=G_{C_i}$, and note that $x_i$ and $x_{i+1}$ are $G_i$-active.

For $i=2,\ldots,m$, choose a color $\psi(x_i)\in L(x_i) \setminus \{\psi(v')\}$. Let $v_1=x_2$ and $v_{m}=x_{m}$.
For $i=2,\ldots,m-1$, let $v_i=x_i$ if $\psi(x_{i+1})=\psi(v)$ and $v_i=x_{i+1}$ otherwise; note that
$\psi(v_i)\neq \psi(v)$ unless $\psi(x_i)=\psi(v)=\psi(x_{i+1})$.
By the induction hypothesis, the restriction of $\psi$ to $V(C_i)$ extends to a weak diameter-$4$ $\psi$-opaque $v_i$-compliant $L$-coloring $\varphi_i$ of $G_i$ for every $i\in \{1,\ldots, m\}$. Let $\varphi$ be the $L$-coloring of $G$ such that $\varphi_i$ is the restriction of $\varphi$ to $G_i$ for every $i$. We show that $\varphi$  satisfies the lemma.

Note that every $v$-compliant coloring of $G$ that extends $\psi$ is necessarily $\psi$-opaque. Thus it suffices to show that  every  monochromatic component  $Q$ of $\varphi$ has weak diameter at most four, and that if $Q$ is $C$-transversal then $Q$ satisfies the conditions (C1)--(C3) above.

Suppose first that $V(Q) \cap V(C) = \emptyset$. If $Q$ is a monochromatic component of $G_i$ for some $i$ then  the weak diameter of $Q$ is at most four by the choice of $\varphi_i$. Thus we may assume that $Q$ contains vertices in both  $V(G_i)\setminus V(G_{i-1})$ and $V(G_{i-1})\setminus V(G_i)$ for some $i \in \{2,\ldots,m\}$. Since $V(Q)\cap V(C)=\emptyset$, it follows that $x_i\in V(C)$ and for $j \in 
\{i-1,i\}$ the restriction of $Q$ to $G_j$ is $C_j$-transversal.  Since $\varphi_{i-1}$ and $\varphi_i$ are $\psi$-opaque, $V(Q)\cap V(C_{i-1}\cup C_i)=\{x_i\}$ and $V(Q)\subseteq V(G_{i-1}\cup G_i)$.
By (C3), every vertex of $Q$ is a neighbor of some vertex in $\{v,v',x_{i-1},x_{i+1}\}$, and thus the weak diameter of $Q$ is at most four as desired.

It remains to consider the case when $Q$ is $C$-transversal.  Consider any edge $uw \in E(G)$
such that $u \in V(C)$, $w \not \in V(C)$ and $\varphi(u)=\varphi(w)$:
\begin{itemize}
\item If $w=x_i$ for some $i\in\{2,\ldots,m\}$, then since $\psi(x_i)\neq \psi(v')$, we have
$u=v$ and $\psi(v) \neq  \psi(v')$.
\item Otherwise, $w\in V(G_i)\setminus V(C_i)$ for some $i\in \{1,\ldots, m\}$.
By (C1) for $\varphi_i$, we have $\psi(v_i)=\psi(u)$, and the choice of $v_1$ and $v_m$ and the
property (C2) of $\varphi_i$ imply $u\not\in \{x_1,x_{m+1}\}$.  Since $\psi(v_i)\in L(v_i)\setminus \{\psi(v')\}$, 
it follows that $u=v$ and $\psi(v)\neq \psi(v')$.
\end{itemize}
In either case, we conclude that $\varphi$ satisfies (C1) and (C2).
It remains to check that (C3) holds, i.e. every vertex $w \in V(Q)\setminus V(C)$ has a neighbor in $V(C)\setminus V(Q)$.
If $w \in X$ then $v'$ is such a neighbor. Hence, assume that  $w \in V(G_i) \setminus V(C_i)$ for some $i$, and thus
the restriction of $Q$ to $G_i$ is $C_i$-transversal.  Recall that $v\in V(Q)$; by (C1) for $\varphi_i$,
we have $v_i\in V(Q)$, and thus $\psi(v_i)=\psi(v)$.  Moreover, (C2) for $\varphi_i$ implies $\psi(x_i)\neq \psi(x_{i+1})$.
It follows from the choice of $v_i$ that $i \in \{1,m\}$, and thus $V(C_i)\setminus V(Q)\subseteq V(C_i)\setminus \{v_i\}\subseteq V(C)$. By (C3) for $\varphi_i$,
the vertex $w$ has a neighbor in $V(C_i)\setminus V(Q) \subseteq V(C)\setminus V(Q)$, as desired.    
 \end{proof}

\section{The non-triangle-free case}\label{sec-3choos}

Next, let us consider the case $c=3$.
Let $S_1$ be a single vertex and $\gamma_1$ the function assigning to this vertex the value $5$,
and let $S_2$ be the 4-cycle with one chord and $\gamma_2$ the function assigning to all its vertices the value $6$.

\begin{lemma}\label{lemma-spars3}
Both $(S_1,\gamma_1)$ and $(S_2,\gamma_2)$ are $3$-sparsifiers of size at most 4.
\end{lemma}
\begin{proof}
Consider $i\in \{1,2\}$, let $L$ be an assignment of lists of size $3$ to vertices of $S_i$, and let $B$
be a multiassigment of a list of size $5$ in case $i=1$ and of list of size $6-\deg_{S_2} v$ to each vertex $v\in V(S_2)$ in case $i=2$.
In the case $i=1$, choose $\varphi(v)\in L(v)$ to be different from the (at most two) colors that appear in $B(v)$ more than once.
Clearly, $\varphi$ is $B$-opaque.

In the case $i=2$, we choose the $L$-coloring of $S_2$ as follows.  Let $R$ consist of the vertices $u\in V(S_2)$
such that $B(u)$ contains at most two distinct colors.  For $u\in R$, let $\varphi(u)\in L(u)$ be chosen different
from the colors in $B(u)$.  For $v\in V(S_2)\setminus R$, let $L'(v)$ consist of the colors in $L(v)$ that appear in $B(v)$ at most once.
Note that if $\deg_{S_2} v=2$, then $|B(v)|\le 4$ and since $B(v)$ contains at least three distinct colors, we have $|L'(v)|\ge 2$;
and if $\deg_{S_2} v=3$, then $|B(v)|\le 3$ and $L'(v)=L(v)$ has size three.  In particular, $|L'(v)|\ge \deg_{S_2} v$ for each
$v\in V(S_2)\setminus R$.  Hence, we can choose $\varphi$ on $V(S_2)\setminus R$ to be a proper $L'$-coloring of $S_2-R$.
Additionally, in case neither of the vertices $x,y\in V(S_2)$ of degree two belongs to $R$ and at least one vertex $z$ of degree
three belongs to $R$ (so $S_2-R$ is either a path or consists of two isolated vertices), we can choose $\varphi$ so that
$\varphi(x)\neq\varphi(y)$.

If distinct vertices $v_1,v_2\in V(S_2)$ both receive the same color $a\in B(v_1)\cap B(v_2)$, then $v_1,v_2\not\in R$
and $v_1v_2\not\in E(S_2)$, and thus $v_1$ and $v_2$ are the vertices of $S_2$ of degree two.  Moreover, by the last condition
in the choice of $\varphi$, since $\varphi(v_1)=\varphi(v_2)$, we have $R=\emptyset$, and thus no other vertex of $S_2$
has color $a$; hence, $v_1$ and $v_2$ do not belong to the same monochromatic component.  It follows that $\varphi$ is $B$-opaque.
\end{proof}

\begin{proof}[Proof of Lemma~\ref{lemma-spaspar} in the case $c=3$]
Let $\varepsilon_3=1/1000$.  We can assume $|V(G)|>4$, as otherwise the claim holds trivially.
Since $G$ has minimum degree at least three, every face of $G$ has length at least three.
Let $\beta=\sum_{f\in F} (|f|-3)$, where the sum is over all faces.  By the generalized Euler's formula,
we have $|E(G)|\le |V(G)|+|F|+g-2$, and since $2|E(G)|=\sum_{f\in F} |f|=3|F|+\beta$,
we conclude that
$$|E(G)|<3|V(G)|-\beta+3g.$$
Since $G$ is simple, does not have separating or non-contractible triangles, and $|V(G)|>4$, every vertex of degree three is incident
with a face of length at least four.  Since $G$ is $(3,4)$-sparsifier-free, $(S_1,\gamma_1)$ has no appearance in $G$,
and thus each vertex of degree at most five is also incident with a face of length at least four.
Hence, the number $n_5$ of vertices of degree at most five is at most $4\beta$.

Let us give each vertex of degree at most five the charge $1$, any vertex of degree $d\ge 6$ charge $d-6$,
and any face $f$ the charge $2|f|-6$.  By the generalized Euler's formula, the sum of charges is
at most $6g-12+4n_3<16\beta+6g$.  Each vertex $v$ of degree $d\neq 6$ now sends $1/8$ to each 
adjacent vertex; this still leaves $v$ with at least $1/8$ units
of charge.  Each face of length at least four sends $1/8$ to each incident vertex, still keeping its charge nonnegative.
After this, each vertex of degree six which received charge sends $1/56$ to each adjacent vertex.
Since each face has non-negative final charge and the total amount of charge did not change, we conclude that the sum of the final charges of vertices is
less than $16\beta+6g$.  Note that each vertex has non-negative final charge, and vertices of degree other than six have
final charge at least $1/8$.

We claim that vertices of degree six have final charge at least $1/56$.  Consider for a contradiction a vertex $v$
of degree six with smaller final charge.  All incident faces must be triangles, only incident
with vertices of degree six, and the faces incident those must also be triangles.
Moreover, since $G$ does not contain separating or non-contractible triangles, the neighbors of $v$ form an induced $6$-cycle.
This is not possible, since $(S_2,\gamma_2)$ does not have an appearance in $G$.

Since every vertex has final charge at least $1/56$, we have $|V(G)|/56\le 16\beta+6g$,
and thus $\beta\ge |V(G)|/1000-g$.
Consequently, $|E(G)|<3|V(G)|-\beta+3g<(3-1/1000)|V(G)|+4g$.
\end{proof}

We now show that Lemma~\ref{lemma-stacks} holds for $c=3$.
Note that in this case, the condition that the resulting coloring is $\psi$-opaque is trivially
satisfied, since $\psi$ cannot have distinct components of the same color on the triangle $C$.
We prove the following stronger statement.
In a coloring $\varphi$ of a 3-stack $G$ with the outer face bounded by a triangle $C$,
a vertex $v\in V(C)$ is a \emph{singleton} if no adjacent vertex in $V(G)\setminus V(C)$ has the color $\varphi(v)$.

\begin{lemma}\label{lemma-stacks3}
Let $G$ be a 3-stack with the outer face bounded by the triangle $C$, and let $u$ be a vertex of $C$.
Let $L$ be an assignment of lists of size three to vertices of $G$, and let $\psi$ be an $L$-coloring of $C$.
Then $\psi$ extends to a weak diameter-$2$ $L$-coloring of $G$ in which vertices in $V(C)\setminus\{u\}$ are singletons
and the monochromatic component containing $u$ is contained in the neighborhood of each of the vertices in $V(C)\setminus \{u\}$;
and moreover, if $\psi$ only uses at most two distinct colors on $C$, then $u$ is also a singleton.
\end{lemma}
\begin{proof}
We prove the claim by induction on the number of vertices of $G$.
The claim is clear if $G=C$, and thus we can assume that there exists a vertex $v\in V(G)$ adjacent to all vertices
of $C$.  Let $C_1$, $C_2$, and $C_3$ be the three triangles in $G[V(C)\cup\{v\}]$ distinct from $C$,
where $u\not\in V(C_3)$.  Choose a color $\psi(v)\in L(v)$ distinct from the colors of the two vertices in $V(C)\setminus \{u\}$,
and distinct from $\psi(u)$ if $\psi$ only uses at most two distinct colors on $C$.

For $i\in\{1,2,3\}$, extend $\psi$ to a weak diameter-$2$ $L$-coloring of $G_{C_i}$ by the induction hypothesis, with the vertex $v$ playing
the role of $u$. This ensures that the vertices in $V(C)\setminus\{u\}$ are singletons in the resulting $L$-coloring of $G$, and
if $\psi(v)\neq \psi(u)$ (which is always the case if $\psi$ only uses at most two colors on $C$), then also $u$ is a singleton.
Consider the monochromatic component $Q$ of the vertex $v$ in the resulting coloring:
\begin{itemize}
\item If $\psi(v)\neq \psi(u)$, then by the induction hypothesis, for $1\le i<j\le 3$, $Q\cap (V(G_{C_i})\cup V(G_{C_j}))$
is contained in the neighborhood of a vertex of $C$, and thus $Q$ has weak diameter at most two.
\item If $\psi(v)=\psi(u)$, then $u$ and $v$ are singletons in the colorings of $G_{C_1}$ and $G_{C_2}$,
and thus $Q=(Q\cap V(G_{C_3}))\cup \{u\}$.  By the induction hypothesis, we conclude that $Q$ is contained in the
neighborhood of each of the vertices in $V(C)\setminus \{u\}$.
\end{itemize}
\end{proof}

\bibliographystyle{siam}
\bibliography{data.bib}

\end{document}